\documentclass{amsart}
\usepackage[applemac]{inputenc}

\usepackage[colorlinks=true,linkcolor=red,citecolor=blue,urlcolor=blue,hypertexnames=false]{hyperref}
\usepackage{bookmark}
\usepackage{amsthm,thmtools,amssymb,amsmath,amscd}

\usepackage{fancyhdr}
\usepackage{esint}
\usepackage{enumerate}

\usepackage{pictexwd,dcpic}

\swapnumbers
\declaretheorem[name=Theorem,numberwithin=section]{thm}

\declaretheorem[name=Lemma,sibling=thm]{lemma}

\declaretheorem[name=Corollary,sibling=thm]{cor}

\numberwithin{equation}{section}

\usepackage{cleveref}
\crefname{lemma}{Lemma}{Lemmata}
\crefname{prop}{Proposition}{Propositions}
\crefname{thm}{Theorem}{Theorems}
\crefname{cor}{Corollary}{Corollaries}
\crefname{defn}{Definition}{Definitions}
\crefname{example}{Example}{Examples}
\crefname{rem}{Remark}{Remarks}
\crefname{assum}{Assumption}{Assumptions}
\crefname{nota}{Notation}{Notation}

\usepackage{autonum}

\newcommand{\cn}{\colon}
\newcommand{\sub}{\subset}
\newcommand{\ov}{\overline}
\newcommand{\mr}{\mathring}

\newcommand{\bbR}{\mathbb{R}}

\newcommand{\8}{\infty}

\newcommand{\al}{\alpha}
\newcommand{\be}{\beta}

\newcommand{\e}{\epsilon}
\newcommand{\ka}{\kappa}

\newcommand{\om}{\omega}

\newcommand{\Si}{\Sigma}
\newcommand{\p}{\varphi}

\newcommand{\vt}{\vartheta}
\newcommand{\Om}{\Omega}
\newcommand{\D}{\Delta}

\newcommand{\T}{\Theta}

\newcommand{\cS}{\mathcal{S}}

\newcommand{\del}{\partial}
\newcommand{\n}{\nabla}
\newcommand{\II}[2]{\mrm{II}\br{#1,#2}}

\newcommand{\rt}{\sqrt}

\newcommand{\pard}[2]{\frac{\partial #1}{\partial #2}}

\newcommand{\ip}[2]{\left\langle #1,#2 \right\rangle}
\newcommand{\fr}[2]{\frac{#1}{#2}}
\newcommand{\x}{\times}

\DeclareMathOperator{\osc}{osc}
\DeclareMathOperator{\const}{const}

\DeclareMathOperator{\Rm}{Rm}
\DeclareMathOperator{\Rc}{Rc}

\DeclareMathOperator{\vol}{vol}

\newcommand{\pf}[1]{\begin{proof}{\parskip\baselineskip{ #1}} \end{proof}}
\newcommand{\eq}[1]{\begin{equation}\begin{alignedat}{2} #1 \end{alignedat}\end{equation}}

\newcommand{\br}[1]{\left(#1\right)}
\newcommand{\abs}[1]{\lvert #1\rvert}
\newcommand{\enum}[1]{\begin{enumerate}[(i)] #1 \end{enumerate}}

\newcommand{\ra}{\rightarrow}

\newcommand{\mrm}{\mathrm}
\newcommand{\bb}[1]{\mathbb{#1}}

\newcommand{\hp}{\hphantom}
\newcommand{\q}{\quad}

\begin{document}

\title[Isoperimetric problems in generalized Robertson-Walker spaces]{Isoperimetric problems for spacelike domains in generalized Robertson-Walker spaces}
\thanks{BL was supported by Leverhulme Trust Research Project Grant RPG-2016-174. The work of JS was funded by the "Deutsche Forschungsgemeinschaft" (DFG, German research foundation); Project "Quermassintegral preserving local curvature flows"; Grant number SCHE 1879/3-1.}

\author[B. Lambert]{Ben Lambert}
\address{Department of Mathematics, University College London,
Gower Street, London WC1E 6BT, UK}
\email{\href{mailto:b.lambert@ucl.ac.uk}{b.lambert@ucl.ac.uk}}

\author[J. Scheuer]{Julian Scheuer}
\address{Department of Mathematics, Columbia University, 2990 Broadway NY 10027, USA}
\email{\href{mailto:jss2291@columbia.edu}{jss2291@columbia.edu}, \href{mailto:julian.scheuer@math.uni-freiburg.de}{julian.scheuer@math.uni-freiburg.de}}
\urladdr{\href{https://home.mathematik.uni-freiburg.de/scheuer/}{https://home.mathematik.uni-freiburg.de/scheuer/}}

\begin{abstract}
We use a locally constrained mean curvature flow to prove the isoperimetric inequality for spacelike domains in generalized Robertson-Walker spaces satisfying the null convergence condition. 
\end{abstract}

\maketitle

\section{Introduction}

For a bounded domain $\Om$ of the two-dimensional Euclidean, hyperbolic or  spherical space there hold respectively
		\eq{L^{2}\geq 4\pi A,\q L^{2}\geq 4\pi A+A^{2},\q L^{2}\geq 4\pi A-A^{2},}
where $L$ is the boundary length and $A$ the area of $\Om$.
Equality holds precisely on balls. 
For hypersurfaces of the $(n+1)$ Euclidean space there holds
\eq{\label{1}\abs{\del\Om}\geq c_{n}\abs{\Om}^{\fr{n}{n+1}},}
with an explicit dimensional constant $c_{n}$, where again, equality holds 	precisely on balls. Here $\abs{\cdot}$ is the Hausdorff measure of a submanifold of the appropriate dimension. Neither in the hyperbolic space nor in the sphere an explicit form like \eqref{1} is available, which holds in all dimensions.
However, the isoperimetric problem is solved, \cite{Rado:10/1935,Schmidt:12/1940b,Schmidt:12/1940}: For any bounded domain $\Om$ in the hyperbolic or spherical space and a ball $B_{r}$ with $\abs{\Om}=\abs{B_{r}}$ there holds
\eq{\abs{\del\Om}\geq\abs{\del B_{r}}} with equality precisely when $\Om=B_{r}$. We can describe this result in an alternative way: Define
\eq{f_{0}(r)=\abs{B_{r}},\q f_{1}(r)=\abs{\del B_{r}},}					then 
\eq{\abs{\del\Om}\geq \abs{\del B_{r}}=f_{1}(r)=f_{1}\circ f_{0}^{-1}(\abs{B_{r}})\equiv\p(\abs{\Om}),}
which is an implicit form of the isoperimetric inequality. In more general Riemannian warped product spaces such an implicit isoperimetric inequality was deduced in \cite{GuanLiWang:/2019}, while also on other Riemannian spaces isoperimetric inequalities have received lots of attention, \cite{BrayMorgan:12/2001,Brendle:07/2019,CorvinoGerekGreenbergKrummel:05/2007,HoffmannSpruck:11/1974,White:/2009}. For the comparison of the volume of a domain in Lorentzian spaces which is bounded by spacelike hypersurfaces, much less seems to be known. Here the question is, which hypersurfaces maximize area under volume constraint. Some results are available in the Minkowski space \cite{BahnEhrlich:06/1999}, on two-dimensional Lorentzian surfaces satisfying a curvature bound \cite{Bahn:07/1999} and in warped product spaces, such as in a certain class of Friedman-Robertson-Walker spaces \cite{AbedinCorvinoKapitanWu:11/2009}.

The goal of this paper is to solve the isoperimetric problem for spacelike domains in a large class of Lorentzian warped product manifolds, which we describe in the following.

A spacetime 
\eq{\label{GRW}N=(a,b)\x \cS_{0},\q \bar g=-dr^{2}+\vt^{2}(r)\hat g,}
with $a<b$ real numbers, a compact $n$-dimensional Riemannian manifold $(\cS_{0},\hat g)$ and positive warping factor $\vt\in C^{\8}([a,b))$ is called {\it{generalized Robertson-Walker space}}.
In this paper we use the locally constrained mean curvature flow
\eq{\label{Div-Flow}\dot{x}=(\D_{\Si_{t}} \T) \nu}
to solve an isoperimetric problem in $N$.
Here $\T'(r)=\vt(r)$, $\T$ is understood to be defined on the ambient manifold and $\nu$ is the future directed timelike (i.e. $\bar g(\nu,\nu)=-1$) normal vector to the flow hypersurfaces $\Si_{t}$.
The terminology used for this flow stems from the fact that there also holds
\eq{\label{Flow}\dot{x}=(uH-n\vt'(\rho))\nu,}
where $H$ is the mean curvature of the flow hypersurfaces with respect to $-\nu$, $\rho(t,\cdot)=r_{|\Si_{t}}$ and $u$ is the support function
\eq{u=-\bar g(\vt\del_{r},\nu).}
The suitable adaption of \eqref{Flow} to the Riemannian setting was introduced in \cite{GuanLi:/2015} and further studied in \cite{GuanLiWang:/2019}. 

A Lorentzian manifold $N$ is said to satisfy the \emph{null convergence condition} if for all lightlike vectors $X$,
\eq{\label{NCC}\ov{\Rc}(X,X)\geq 0.}
We also say that the null convergence condition is satisfied \emph{strictly}, if
\eqref{NCC} holds with strict inequality for all nonzero lightlike $X$. 
We observe that this condition is implied by the more commonly used timelike convergence condition which is well known to be important in prescribed mean curvature equations, see for example the work of Bartnik \cite{Bartnik:/1984}, Ecker and Huisken \cite{EckerHuisken:/1991} and Ecker \cite{Ecker:/1997}. However, it is also valid on any Einstein manifold, while the timelike convergence condition is not.

We prove the following isoperimetric inequality for domains bounded by a time-slice and a closed spacelike hypersurface. 

\begin{thm}\label{main}
Let $n\geq 2$ and $N^{n+1}$ be a generalized Robertson-Walker space which satisfies the null convergence condition. Let $\Si\sub N$ be a spacelike, compact, achronal and connected hypersurface. Then there holds
\eq{\p(\vol(\hat\Si))\geq \abs{\Si},}
where $\p\cn [0,\vol(N))\ra \bbR_{+}$ is the function which gives equality on the coordinate slices. Furthermore:
\enum{
\item
If equality holds, then $\Si$ is totally umbilic. 
\item If $N$ satisfies the null convergence condition strictly, then equality is attained precisely on the timeslices of $N$.
}
\end{thm}

A hypersurface $\Si\sub N$ is called spacelike, if the induced metric is positive definite. $\hat \Si$ is the region enclosed by the horizon $\{a\}\x \cS_{0}$ and $\Si$, see \Cref{Def}. Such a spacelike hypersurface is called achronal if no timelike curve meets $\Si$ more than once (see \cite[p.~425]{ONeill:/1983} for more details). Also note that this is automatically satisfied if $N$ is simply connected, \cite[p.~427]{ONeill:/1983}.

\section{Conventions and some hypersurface geometry}\label{Def}

\subsection{Basic notation}
Throughout this paper we use the curvature conventions from \cite{ONeill:/1983}, in particular the Riemann tensor of a semi-Riemannian manifold with metric $\bar g$ and Levi-Civita connection $\ov\n$ is defined by
\eq{\ov \Rm(X,Y)Z=\ov\n_{Y}\ov\n_{X}Z-\ov\n_{X}\ov\n_{Y}Z-\ov\n_{[Y,X]}Z}
for all vector fields $X,Y,Z$ on $N$. 

Given any orthonormal frame $E_1, \ldots, E_{n+1}$ where $E_{n+1}$ is timelike, define the Ricci curvature by
\eq{\ov{\Rc}(X,Y) = \bar g(\ov\Rm(E_i,X)E^i,Y) - \bar g(\ov\Rm(E_{n+1},X)E_{n+1},Y).}
Here the summation has been chosen so that the Ricci curvature of the Lorentz product metric on $\bb{S}^n\times \bb{R}$ is non-negative.

\subsection*{Spacelike hypersurfaces}
 
Let $\Si\sub N$ be a spacelike, compact, connected and achronal hypersurface given by an embedding $x$. The manifold $N$ is globally hyperbolic \cite[Thm.~1.4.2]{Gerhardt:/2006} and $\cS_0$
 is a Cauchy hypersurface. Thus $\Si$ is a graph over $\cS_{0}$, 
\eq{\Si=\{(\rho(x^{i}),x^{i})\cn (x^{i})\in \cS_{0} \},}
see \cite[Prop.~1.6.3]{Gerhardt:/2006}.
Latin indices range between $1$ and $n$ and greek indices range from $0$ to $n$. Sometime we will write 
\eq{x^{0}=r.}

We state the relations between the geometric quantities of $\Si$ and the graph function $\rho$. Details can be found in \cite[Sec.~1.6]{Gerhardt:/2006}.
 We use the coordinate based notation, e.g. the induced metric $g$ is 
\eq{g_{ij}=g(\del_{i},\del_{j})}
and we denote its Levi-Civita connection by $\n_{i}=\n_{\del_{i}}$. We also write
\eq{x_{i}:=\del_{i}x.}

Let $\nu$ be the future directed timelike normal, i.e.
\eq{\bar g(\del_{r},\nu)<0,}
and define the shape operator $S=(h^{i}_{j})$ of $\Si$ with respect to $-\nu$.
Then we call
\eq{A=h_{ij}:=g_{ik}h^{k}_{j}=-\bar g\br{\II{\del_{i}}{\del_{j}},\nu}=\ov{g}(\ov \n_{i}\nu, x_{j}).}
 the second fundamental form of $\Si$, which has eigenvalues with respect to $g$ ordered by
 \eq{\ka_{1}\leq\dots\leq \ka_{n}.}
 For any spacelike hypersurface the Codazzi--Mainardi equations may be written \cite[Prop.~33, p.~115]{ONeill:/1983}
\eq{\label{Codazzi}\bar g(\ov{\Rm}(X,Y)Z,\nu)=\n_Xh(Y,Z) - \n_Yh(X,Z).}

The second fundamental form of the slices $\{x^{0}=r\}$ is
\eq{\bar{h}_{ij}:=\fr{\vt'(r)}{\vt(r)}\bar g_{ij},}
\cite[(1.6.13)]{Gerhardt:/2006},
while the induced metric is
\eq{g_{ij}=-\del_{i}\rho \del_{j}\rho+\vt^{2}(\rho)\hat g_{ij}=-\del_{i}\rho\del_{j}\rho+\bar g_{ij}.}
With the definition
\eq{v^{2}=1-\vt^{-2}\hat{g}^{ij}\del_{i}\rho\del_{j}\rho,}
the second fundamental form satisfies
\eq{\label{h}v^{-1}h_{ij}=\n_{ij}\rho+\bar h_{ij},}
\cite[(1.6.11)]{Gerhardt:/2006}. Note that in this reference the past directed normal is used. 

Suppose $\T$ solves $\T'(r)=\vt(r)$, and by abuse of notation we identify $\T=\T(r)$ so that $\T:N\ra \bb{R}$. 
As a function on $N$ we have
\eq{\ov{\n}^2_{\alpha\beta} \T = -\vt' \ov{g}_{\alpha\beta},}
while on $\Sigma_t$, 
\eq{\label{rho}\n_{ij}\T(\rho)=\vt' \del_{i}\rho\del_{j}\rho+\vt \n_{ij}\rho=-\vt' g_{ij}+\fr{\vt}{v}h_{ij}.}
In particular we observe that \eqref{Div-Flow} and \eqref{Flow} are the same flows.

We define the support function 
\eq{u:=\fr{\vt}{v}=-\bar g(\vt\del_{r},{\nu}) = \ov{g}(\ov \n \T, \nu),}
and observe that this is related to $\n\T$ by the identity
\eq{|\n \T|^2\equiv g^{ij}\del_{i}\T\del_{j}\T= u^2-\vt^2.}

We will use the following important inequality in several places.
\begin{lemma} If $N$ satisfies the null convergence condition and $\Si\sub N$ is a spacelike graph as above, then
\eq{\label{TCCuse}\ov{\Rc}(\n \Theta, \nu)\geq 0.}
\end{lemma}
\pf{In a GRW space $\ov\n \T$ is an eigenvector of $\ov{\Rc}$, \cite[Cor.~43, p.~211]{ONeill:/1983}, and so
\eq{\label{TCCuse-1}
 \ov{\Rc}(\nu, \n \T) &=\ov{\Rc}(\nu, \ov \n \T + u \nu)\\
 &=u\ov{\Rc}(\nu,\nu)+\ov{\Rc}(\nu, \ov \n \T)\\
 &=u\ov{\Rc}\left(V-\frac{\ov \n \T}{\vt} \frac{u}{\vt},V-\frac{\ov \n \T}{\vt} \frac{u}{\vt}\right)+\ov{\Rc}\left(V-\frac{\ov \n \T}{\vt} \frac{u}{\vt}, \ov \n \T\right)\\
 &=u\left[\left(\frac{u^2}{\vt^2}-1\right)\ov{\Rc}\left(\frac{\ov \n \T}{\vt},\frac{\ov \n \T}{\vt}\right)+\ov{\Rc}\left(V,V\right)\right]\\
&=u\ov{\Rc}(W,W),
 }
where $V$ is the projection of $\nu$ onto $ (\ov \n \T)^\perp$ and
\eq{\label{TCCuse-2}W=V+\sqrt{\frac{u^2}{\vt^2}-1}\,\frac{\ov \n \T}{\vt}.}
 Since $\left|V\right|^2 = u^2\vt^{-2}-1$, $W$ is a lightlike vector and the result follows. 
}

\subsection*{Area and volume calculations}
Let $\Si$ be graphical as above. We define the integral of a function $f\in C^{\8}(\Si)$ by
\eq{\int_{\Si}f:=\int_{\mathcal{S}_0}
f\,d\om_{g}.}
Here $d\om_{g}$ is the Riemannian volume form on $\Si$. 
For
\eq{\hat\Si:=\{(r,\xi)\in N\cn a\leq r\leq \rho(\xi),\, \xi\in\cS_{0}\}}
we define the {\it{enclosed volume}} (see \cite[p.~194]{ONeill:/1983}) by
\eq{\label{vol}\vol(\hat\Si):=\int_{\cS_{0}}\int_{a}^{\rho(\cdot)}\fr{\rt{\det(\bar g_{ij}(s,\cdot))}}{\rt{\det(\bar g_{ij}(a,\cdot))}}\,ds\,d\omega_a=\int_{\hat\Si}d\vol_{\bar g},}
where $d \omega_a$ is the volume form on the time slice $\{a\}\times\mathcal{S}_0$ and locally
\eq{d\vol_{\bar g}=\rt{\abs{\det (\bar g_{\al\be})}}.}
The surface area of $\Si$ is
\eq{\abs{\Si}=\int_{\Si}1.}
Suppose $\hat{\Sigma}\subset N$ is open with compact closure such that $\partial \hat{\Sigma}$ may be written as a union of smooth spacelike hypersurfaces with \emph{outward pointing} normal $\nu$. If $X$ is a smooth vector field on $\hat{\Sigma}$ then
\eq{\int_{\hat\Sigma} \ov{\operatorname{div}} X\, d\vol = -\int_{\partial \hat\Si}\ip{\nu}{X}.}
This follows from Stoke's theorem and \cite[Lemma~21, p.~195]{ONeill:/1983}. 

We now suppose that $\hat{\Sigma}\subset N$ is a time dependent set which is bounded by spacelike hypersurfaces $\Sigma_0$ and $\Sigma$, where $\Sigma$ varies with time and $\Sigma_0$ is fixed. Let $x$ be a time dependent parametrization of $\Sigma$ then the above divergence theorem and \cite[Lemma~21, p.~195]{ONeill:/1983} imply that
\eq{\label{volchange}\partial_t\vol(\hat \Si_t) = -\int_{\Si}\ip{\dot{x}}{\nu}.}

\begin{lemma}
\label{AreaMonotone}
Let $n\geq 2$. Along the flow \eqref{Div-Flow}
\enum{
\item
the volume $\vol(\hat\Si)$ is preserved and
\item the surface area increases, provided $(N,\bar g)$ satisfies the timelike convergence condition.
}
If $\del_{t}|\Sigma_t|=0$, then $\Sigma$ is umbilic.
\end{lemma}
\pf{
By equations \eqref{volchange} and \eqref{Div-Flow} we have that
\eq{\del_{t}\vol(\hat\Si_{t})=\int_{\Si_{t}}\Delta \T=0.}

We recall that
\eq{ \sigma_2=\fr{1}{2}\br{H^{2}-\abs{A}^{2}}, \qquad \sigma_2^{ij} = Hg^{ij} - h^{ij},}
so
\begin{align}
\n_i\sigma_2^{ij} &= \n^j H - \n_i h^{ij}=-\bar g(\ov{\Rm}(x_i, x^j)x^i,\nu)=-\ov{\Rc}(x^j,\nu).
\end{align}
Therefore, by the divergence theorem
\eq{\label{tracesigmatwo}\int_{\Si_{t}} \br{2\sigma_2u - (n-1)\vt'H} =\int_{\Si_{t}} \sigma_2^{ij} \n^2_{ij}\T =\int_{\Si_{t}} \ov{\Rc}(\nu,\n\T).}
Using Lemma \ref{evolgijGRW} we get
\eq{\label{Ev-area}
\del_t|\Sigma_t| 
&= \int_{\Si_{t}} (H^2u - n\vt' H)\\
&=\int \br{H^2 - \frac{2n}{n-1}\sigma_2+\fr{n}{u(n-1)}\ov{\Rc}(\n \T, \nu)}u\\
&\geq \fr{n}{n-1}\int_{\Si_{t}} \abs{\mr A}^{2}u.
}
where we used \eqref{tracesigmatwo} on the second line and \eqref{TCCuse} on the last line. 
}


\section{Evolution equations}
We now calculate several required evolution equations.

\begin{lemma}On $\Sigma_t$ the function $\T $ satisfies \label{dfevolGRW}
\eq{\dot \T -u\Delta \T  =0.}
\end{lemma}
\begin{proof}
We calculate that
\begin{align}
\dot{\T } &= \bar g(\ov\n \T,\dot x)=\bar g(\ov\n \T,\nu)\Delta \T  = u\Delta\T.\end{align}

\end{proof}

%
%

\begin{lemma}\label{evolgijGRW}
On $\Sigma_t$ the induced metric $g_{ij}$ satisfies
\eq{\dot g_{ij} =2h_{ij}\Delta \T  }
\end{lemma}
\begin{proof}
\begin{align}
\partial_t\bar g(x_i,x_j)&=\bar g\br{\pard{}{x^i}(\Delta \T \nu),{x_j}}+\bar g\br{\pard{}{x^j}(\Delta \T \nu),x_i}=2h_{ij}\Delta \T. 
\end{align}
\end{proof}

\begin{lemma}On $\Sigma_t$ the future oriented normal $\nu$ satisfies
\begin{align}
\ov\n_{\dot x}\nu &= \n\Delta \T = u\n H+H\n u -n\frac{\vt''}{\vt}\n\T.
\end{align}
\end{lemma}
\pf{
We have 
\begin{align}
\bar g\br{\ov\n_{\dot x} \nu,x_i}=-\bar g\br{\nu,\ov\n_{\dot x} x_i} = -\bar g\br{\nu,\ov{\n}_{x_i} \dot{x}} = -\bar g\br{\nu,\ov{\n}_{x_i} (\Delta \T  \nu)} = \n_{i} \Delta \T 
\end{align}
and observe  
\eq{\n_i(-n\vt')= -n \vt'' \bar g(\ov\n r,x_i)= -n \frac{\vt''}{\vt} \bar g(\ov\n \T,x_i).}
}

\begin{lemma}On $\Sigma_t$ the function $u$ satisfies

\eq{\dot{u}-u\D u&=-\abs{\mr A}^{2}u^{2}-\br{\fr{Hu}{\rt n}-\rt{n}\vt'}^{2}-n\fr{\vt''}{\vt}(u^2-\vt^2)\\
&\hp{=}-u\ov{\Rc}(\n\T,\nu)+H\bar g(\n u,\n\T).}



\end{lemma}
\begin{proof}
We calculate
\begin{align}
\dot{u}&=\bar g\br{\ov\n_{\dot x}\nu,\ov\n \T }+\bar g\br{\nu,\ov\n_{\dot x}\ov\n \T }\\
&=u\bar g(\n H,\n \T)+H\bar g(\n u,\n \T)-n\frac{\vt''}{\vt}|\n\T|^2+\Delta \T   \ov\n^2_{\nu\nu} \T.
\end{align}
We also see that
\[\n_i u =  h_i^k\n_k\T +\ov\n^2_{i\nu} \T  =h_i^k\n_k\T\]
and
\begin{align}
\n^2_{ij} u&=\n_jh_{ik}\n^k\T +h_i^kh_{kj}u +h_i^k\ov\n^2_{x_kx_j}\T =\n_jh_{ik}\n^k\T +h_i^kh_{kj}u -\vt'h_{ij}.
\end{align}
Taking a trace, and applying Codazzi Mainardi on the first term we see that
\begin{align}
\Delta u &=\bar g(\n H,\n\T) +\bar g(\ov{\Rm}(x_j,\n \T)x^j,\nu)+|A|^2u-\vt'H\\
&=\bar g(\n H,\n\T)+|A|^2u+\ov{\Rc}(\n \T,\nu) -\vt'H. \label{Deltau}
\end{align}
Overall we have
\eq{
&\dot{u}-u\D u\\ 
=~&-|A|^2u^2+H\bar g(\n u,\n \T)-n\frac{\vt''}{\vt}(u^2-\vt^2)+\vt'\Delta \T 
-u\ov{\Rc}(\n \T,\nu) +\vt'uH\\
=~&-|A|^2u^2-u\ov{\Rc}(\n \T,\nu)+2\vt'uH -n\frac{\vt''}{\vt}u^2 +n(\vt''\vt-(\vt')^2)+H\bar g(\n u,\n \T)\\
=~&-\abs{A}^{2}u^{2}+\fr{1}{n}H^{2}u^{2}-\fr{1}{n}H^{2}u^{2}+2\vt'Hu-n\vt'^{2}-n\fr{\vt''}{\vt}u^{2}+n\vt''\vt\\
-&~u\ov{\Rc}(\n\T,\nu)+H\bar g(\n u,\n\T)\\
=~&-\abs{\mr A}^{2}u^{2}-\br{\fr{Hu}{\rt n}-\rt{n}\vt'}^{2}-n\fr{\vt''}{\vt}u^{2}+n\vt''\vt-u\ov{\Rc}(\n\T,\nu)+H\bar g(\n u,\n\T).
}
\end{proof}

\section{Gradient estimate}


\begin{lemma}\label{dfest}
There exist uniform bounds
\eq{\inf_{\Sigma_0} \T\leq\T(p, t)\leq\sup_{\Sigma_0} \T}
\end{lemma}
\begin{proof}
This follows directly from Lemma \ref{dfevolGRW} and the maximum principle.
\end{proof}

\begin{lemma}[Gradient bound]\label{ubound}
The support function is uniformly bounded along the flow.
\end{lemma}

\pf{

Define
\eq{w=u+\T^2.}
Then $w$ satisfies
\eq{\dot{w}-u\D w&\leq -\fr{H^{2}u^{2}}{n}+2Hu\vt'+H\bar g(\n u,\n\T)-2u|\n\T|^2+c(u^{2}+1)\\
&\leq-\frac{H^{2}u^{2}}{n}+2Hu\vt'-2H\T(u^2-\vt^2)-2u^3\\&\hp{=}+c(u^{2}+1)+H\bar g(\n w,\n\T)\\
&\leq\left(\e-\frac 1 n\right)H^{2}u^{2}-2u^3+c_\e(u^{2}+1)+H\bar g(\n w,\n\T),\\
}
where we estimated using Young's inequality on the final line. At a large maximal point of $w$, $u$ must also be very large, as $\rho$ is bounded. Setting $\e=\frac 1 n$, the result follows from the maximum principle.

}

\begin{cor}\label{Reg}
Along \eqref{Flow} we have uniform $C^{m}$-estimates for every $m$ and long-time existence of the flow.
\end{cor}

\pf{
Under \eqref{Flow}, the graph function $\rho(\cdot,t)$ satisfies a quasi-linear parabolic equation which, by Lemmas \ref{dfest} and \ref{ubound}, is uniformly parabolic. Lemmas \ref{dfest} and \ref{ubound} and standard application of the Nash--Moser--Di Giorgi theorem \cite[Ch.~XII]{Lieberman:/1998} provides uniform $C^{1+\alpha;\frac{1+\alpha}{2}}$ bounds on $\rho$, and then Schauder theory \cite[Thm.~IV.10.1, p.~351-352]{LadyzenskayaSolonnikovUraltseva:/1968} implies uniform estimates to all orders. Standard parabolic existence theory completes the proof.
}

\section{Completion of the proof}

\pf{
We have to prove that the flow converges to a coordinate slice $\{r=\const\}$ and finish the proof of \Cref{main}. We will prove that the function $\T$ converges to a constant as $t\ra\infty$ using similar methods to \cite[Thm.~3.1]{AltschulerWu:01/1994} and
\cite[Sec.~6.2]{Schnurer:07/2002}. Recall that $\T$ satisfies
\eq{\dot{\T}-u\D \T=0,}
where $u\D_{\Si_{t}}$ is uniformly elliptic due to the support function estimates. Hence $\T$ enjoys the validity of the strong maximum principle for parabolic operators and hence the oscillation of $\T$,
\eq{\om(t)=\osc\T(t,\cdot)=\max \T(t,\cdot)-\min\T(t,\cdot)}
is strictly decreasing, unless $\T$ is constant at some (and hence all) $t>0$, in which case we would be done.

Suppose that $\om$ does not converge to zero as $t\ra \8$. Then it converges to another value $\om_{\8}>0$. Define a sequence of flows by
\eq{x_{i}(t,\xi)=x(t+i,\xi)}
and the corresponding functions $\T_{i}$.
Due to \Cref{Reg}, on a given time interval $[0,T]$ we can apply Arzéla-Ascoli and obtain smooth convergence of a subsequence of $x_{i}$ to a limit flow
\eq{x_{\8}\cn [0,T]\x \cS_{0}\ra N,}
which solves the same flow equation \eqref{Div-Flow}. By construction, the oscillation of the associated limiting function $\T_{\8}$ is $\om_{\8}>0$ constantly, which is a contradiction to the strong maximum principle, which holds for $\T_{\8}$ as well.
We conclude that 
\eq{\lim_{t\ra \8}\om(t)=0}
and hence every subsequential limit of the original flow $x$ must be a time-slice of the spacetime $N$. By the barrier estimates in Lemma \ref{dfest}, this timeslice is unique and we obtain that the whole flow $x$ converges to a timeslice.

We conclude the proof by showing that the isoperimetric inequality holds. Hence let $\Si$ satisfy the assumption of \Cref{main} and evolve $\Si$ by the flow \eqref{Div-Flow}. Define
\eq{S_{R}=\{r=R\},\q f_0(R)=\vol(\hat S_{R}),\q f_1(R)=\abs{S_{R}}.}
Clearly $f_0$ is monotonically increasing in $R$, and $\varphi = f_1\circ f_0^{-1}$. As $\vol(\Sigma_t)$ is fixed, this defines a unique slice $S_{R_\8}$ to which the flow must converge with area $\varphi(\vol(S_{R_{\8}}))$. By the monotonicity properties of Lemma \ref{AreaMonotone} the claimed isoperimetric inequality holds.



If equality holds and $\Si$ was not umbilic, then equation \eqref{Ev-area} implies that variations of $\Si$ along \eqref{Div-Flow} would violate this inequality. Hence in the equality case $\Si$ must be umbilic. 

It remains to prove item (ii) of \Cref{main}. On a time slice equality holds by construction. Hence assume equality holds on $\Si$ and evolve $\Si$ by \eqref{Div-Flow}. The variation formula for the area \eqref{Ev-area} and \eqref{TCCuse-1} show that
\eq{\del_{t}\abs{\Si_{t}}\geq \frac{n}{n-1}\int_{\Si_{t}}u\ov{\Rc}(W,W). }
Hence we must have
\eq{\ov{\Rc}(W,W)=0,}
for otherwise we would reach a contradiction to what we have already proved. Due to the strict null convergence condition we obtain $W=0$ and  from \eqref{TCCuse-2} we deduce
\eq{0=\rt{\fr{u^{2}}{\vt^{2}}-1}=\rt{\fr{1}{v^{2}}-1},}
hence $v=1$ and $\n\T=0$. This shows that $\Si$ is a timeslice.
}

\section*{Acknowledgments}
This work was made possible through a research scholarship JS received from the DFG and which was carried out at Columbia University in New York. JS would like to thank the DFG, Columbia University and especially Prof.~Simon Brendle for their support.

The authors would like to thank Prof. Mu-Tao Wang for an interesting discussion.

\bibliographystyle{shamsplain}
\bibliography{Bibliography.bib}

\end{document}